\documentclass[12pt]{amsart}
\usepackage{amsopn,amsfonts,amssymb,amsthm,mathtools,amsmath}
\usepackage{latexsym}

\usepackage{wrapfig}
\usepackage[T1]{fontenc}
\usepackage{lmodern}
\usepackage{textcomp}
\usepackage[hidelinks]{hyperref}
\usepackage[utf8]{inputenc}
\usepackage{enumerate}
\usepackage[shortlabels]{enumitem}
\setlist[enumerate, 1]{\sc(1)}
\usepackage{float}
\usepackage{hhline}
\usepackage{multirow}
\usepackage{anysize}
\marginsize{3cm}{3cm}{3cm}{3cm}
\usepackage{graphicx}
\usepackage{enumerate}

\def\z{\mathfrak{z}}
\def\u{\mathfrak{u}}
\def\b{\mathfrak{b}}

\def\g{\mathfrak{g}}

\def\e{\mathfrak{e}}

\def\so{\mathfrak{so}}

\def\X{\mathfrak{X}}

\def\A{\mathcal{A}}
\def\B{\mathcal{B}}
\def\G{\mathcal{G}}
\def\l{\ell}

\def\C{\mathbb{C}}
\def\R{\mathbb{R}}
\def\Q{\mathbb{Q}}
\def\Z{\mathbb{Z}}
\def\N{\mathbb{N}}

\newcommand{\ort}{{\sf O}}
\newcommand{\GL}{{\sf GL}}

\def\ad{\operatorname{ad}}

\def\tr{\operatorname{tr}}
\def\Id{\operatorname{Id}}

\def\Lie{\operatorname{Lie}}

\newcommand{\Aut}{\operatorname{Aut}}
\newcommand{\Ker}{\operatorname{Ker}}

\DeclareMathOperator{\Iso}{Iso}
\DeclareMathOperator{\im}{Im}
\DeclareMathOperator{\ord}{ord}
\DeclareMathOperator{\hol}{Hol}
\DeclareMathOperator{\lcm}{lcm}
\DeclareMathOperator{\lie}{Lie}
\def\alt{\raise1pt\hbox{$\bigwedge$}}

\def\la{\langle}
\def\ra{\rangle}

\theoremstyle{plain}
\newtheorem{teo}{\bf Theorem}[section]
\newtheorem{coro}[teo]{\bf Corollary}
\newtheorem{prop}[teo]{\bf Proposition}
\newtheorem{lema}[teo]{\bf Lemma}

\theoremstyle{definition}
\newtheorem{defi}[teo]{\bf Definition}

\theoremstyle{remark}
\newtheorem{obsx}[teo]{Remark}
\newenvironment{obs}
{\pushQED{\qed}\obsx}
\newcommand{\ri}{{\rm (i)}}
\newcommand{\rii}{{\rm (ii)}}
\newcommand{\riii}{{\rm (iii)}}

\newcommand{\norm}[1]{\left\lVert#1\right\rVert}
\newcommand{\matriz}[1]{\ensuremath{\begin{pmatrix}#1\end{pmatrix}}}

\newcommand{\comillas}[1]{``#1''}

\title{Holonomy groups of compact flat solvmanifolds}

\author{A. Tolcachier}
\email{atolcachier@famaf.unc.edu.ar}

\date{}
\address{FaMAF, Universidad Nacional de C\'ordoba, Ciudad Universitaria, X5000HUA C\'ordoba, Argentina}
\thanks{}

\dedicatory{}

\begin{document}

\renewcommand{\bibname}{References}

\begin{abstract}
In this article we study the holonomy group of flat solvmanifolds. It is known that the holonomy group of a flat solvmanifold is abelian; we give an elementary proof of this fact and moreover we prove that any finite abelian group is the holonomy group of a flat solvmanifold. Furthermore, we show that the minimal dimension of a flat solvmanifold with holonomy group $\Z_n$ coincides with the minimal dimension of a compact flat manifold with holonomy group $\Z_n$. Finally, we give the possible holonomy groups of flat solvmanifolds in dimensions 3, 4, 5 and 6; exhibiting in the latter case a general construction to show examples of non cyclic holonomy groups.
\end{abstract}

\maketitle

\tableofcontents

\section{Introduction}
A solvmanifold is defined as a compact homogeneous space $G\backslash\varGamma$ of a simply connected solvable Lie group $G$ by a discrete subgroup $\varGamma$. Solvmanifolds generalize the well known nilmanifolds which are defined similarly when $G$ is nilpotent. Both nilmanifolds and solvmanifolds have provided a large number of examples and counterexamples in differential geometry . For instance, the first example of a symplectic manifold without Kähler structure, the so-called \comillas{Kodaira-Thurston manifold}, is a four dimensional nilmanifold  \cite{Th}. However, many important global properties of nilmanifolds cannot be generalized to solvmanifolds, and for this reason these manifolds are currently widely studied. One of the most important features of nilmanifolds is a result proved by Nomizu, which says briefly that the cohomology of the nilmanifolds can be computed using invariant differential forms, since it is isomorphic to the cohomology of its Lie algebra \cite{No54}. Unfortunately this is not true in general for solvmanifolds, but only in particular cases. For example, when $G$ is completely solvable \cite{Hat} or when the \comillas{Mostow condition} holds \cite{Mo61}. %In order to study nilmanifolds and solvmanifolds, we need previously to construct them. In particular it is in general not easy to understand when a discrete subgroup of a Lie group is a lattice. Again for nilmanifolds there exists a complete theory. Indeed, Malcev's Theorem assures that a nilpotent Lie group admits a lattice if and only if the structure constants of the associated Lie algebra are rational \cite{Ma}. As for many other features, also the existence of a lattice is not as much easy to find for a general solvable Lie group. Fortunately if the solvable Lie group is almost abelian there exists a method to construct lattices \ref{bock}.

Solvmanifolds have had several applications in complex geometry. For instance let us mention the well known Oeljeklaus-Toma manifolds. In \cite{OT}, for any integers $s>0$ and $t>0$, Oeljeklaus and Toma constructed compact complex manifolds by using number theory and they showed that for $t=1$ these manifolds admit locally conformally Kähler metrics; they provide a counterexample to a conjecture made by Vaisman. Kasuya proved in \cite{K} that OT-manifolds can be represented as solvmanifolds and using this description he proved that they do not admit any Vaisman metric. 

Solvmanifolds have also had applications in theoretical physics, for instance in \cite{FMOU}, solutions of the supersymmetric equations of type II A are exhibited on six-dimensional solvmanifolds admitting an invariant symplectic half-flat structure.
In \cite{OUV} the authors provide an example of a six-dimensional complex solvmanifold which admits many invariant solutions to the Strominger system with respect to certain family of Hermitian connections in the Green-Schwarz anomaly cancelation condition. Moreover, some of these solutions satisfy in addition the heterotic equations of motion. 

We are interested in flat solvmanifolds, i.e, solvmanifolds admitting a flat Riemannian metric induced by a flat left invariant metric on the corresponding Lie group. Milnor \cite{Mi} gave a characterization of Lie groups which admit a flat left invariant metric and he showed that they are all solvable of a very restricted form, proving that its Lie algebra decomposes orthogonally as an abelian subalgebra and an abelian ideal, where the action of the subalgebra on the ideal is by skew-adjoint endomorphisms. Moreover, such a Lie group $G$ is of type (I), that is the eigenvalues of $\ad_x$ are pure imaginary for all $x\in\Lie(G)$, and in particular $G$ is unimodular. Some simply connected Lie groups of this class admit lattices (i.e. cocompact discrete subgroups), therefore the corresponding solvmanifolds admit a flat Riemannian metric induced by the flat left invariant metric on the Lie group.

Flat solvmanifolds are a special class of compact flat manifolds. A classical theorem asserts that a complete Riemannian manifold is flat if and only if its universal covering is isometric to Euclidean space. In particular, the manifold is isometric to a quotient of the form $\R^n/\Gamma$ for some discrete subgroup $\Gamma$ of $\Iso(\R^n)$, and its fundamental group is isomorphic to $\Gamma$. In the compact case, such subgroups are well described by the well known Bieberbach's theorems (\cite{Bi1,Bi2,Bi3}) and consequently they are called \comillas{Bieberbach groups}. Particularly, a Bieberbach group $\Gamma$ admits a unique maximal normal abelian subgroup $\Lambda$ of finite index. Furthermore, the finite group $\Gamma/\Lambda$ can be identified with the Riemannian holonomy group of the compact flat manifold. A famous result of Auslander and Kuranishi \cite[Theorem 3]{AusK} establishes that every finite group is the holonomy group of a compact flat manifold. 

The fundamental group of a solvmanifold $\varGamma\backslash G$ is isomorphic to the lattice $\varGamma$ and when the solvmanifold is flat it is isomorphic to a Bieberbach group. In \cite{Aus}, fundamental groups of flat solvmanifolds were characterized.

The purpose of this article is to study the holonomy group of flat solvmanifolds. It was proved in \cite{Aus} that these finite holonomy groups are abelian. In \S3, by using Milnor's result mentioned above, we provide an elementary proof of this fact (Theorem \ref{holabeliana}). The main result in this section is the proof of the converse. In order to do this, we focus on a special class of flat solvmanifolds which are obtained as quotients of almost abelian Lie groups, since for these Lie groups there exists a criterion to determine the existence of lattices (see \cite{B}). We show that the holonomy group of a flat almost abelian solvmanifold is cyclic (Theorem \ref{holonomia_lcm}), and furthermore for $n=p^k$ and $n=2p^k$, with $p$ a prime and $k\in\N$, we give an explicit construction of an almost abelian Lie group and a lattice such that the quotient solvmanifold has holonomy group $\Z_n$ (Theorem \ref{holpotprimos}). Combining  these results we prove the converse, namely, that every finite abelian group is the holonomy group of a flat solvmanifold (Theorem \ref{holabelianofinito}). 

Due to the result of Auslander and Kuranishi, it is interesting to know, given a finite group $G$, the minimal dimension of a compact flat manifold with holonomy group $G$. The minimal dimension was determined for several groups, for example for $\Z_p$ with $p$ prime in \cite{Ch}, and more generally for cyclic groups in \cite{Hi}. Concerning this, we prove in \S4 that the minimal dimension of a flat solvmanifold with holonomy group $\Z_n$ coincides with the minimal dimension of a compact flat manifold with holonomy group $\Z_n$, for $n\geq 3$ (Theorem \ref{mindim}).

The last section of the article (\S5) is devoted to studying holonomy groups of low-dimensional flat solvmanifolds. By using a refined version of Milnor's result proved by Barberis, Dotti and Fino in \cite{BDF}, we give the possible holonomy groups of flat solvmanifolds in dimensions 3, 4, 5 and 6. Up to dimension 5 all flat solvmanifolds are quotients of almost abelian Lie groups, therefore their holonomy groups are cyclic.  On the other hand, in dimension 6 we exhibit examples with non-cyclic abelian holonomy groups as particular cases of a general construction.

The results of this article are from part of my bachelor thesis under the direction of Adrian Andrada. I would like to thank Roberto Miatello and Juan Pablo Rossetti for their useful suggestions and comments.

\medskip

\section{Preliminaries}
In this short section we outline the basic theory  of the topics studied in this article, namely compact flat manifolds and solvmanifolds equipped with invariant flat metrics.

\subsection{Bieberbach groups and compact flat manifolds}
It is well known that the group of isometries of $\R^n$ is $\Iso(\R^n)\cong\ort(n)\ltimes_\varphi\R^n$, where $\varphi(A)(v)=Av$. The pair $(A,v)$ is identified with the isometry $f$ of $\R^n$ given by $f(x)=Ax+v$. This group is called the \textit{group of rigid transformations} of $\R^n$. Equipped with the product topology of $\ort(n)$ and $\R^n$, $\Iso(\R^n)$ is a topological group. The \textit{affine group} $\A_n$ is defined by $\A_n=\GL(n,\R)\ltimes_\varphi\R^n$. An element $(A,v)\in\Iso(\R^n)$ is called a \textit{pure translation} if $A=\Id$. If $\Gamma$ is a subgroup of $\Iso(\R^n)$, the subgroup $\Gamma\cap\R^n$ of pure translations of $\Gamma$ will be usually denoted by $\Lambda$, and it is a torsionfree abelian normal subgroup of $\Gamma$.

\begin{defi}
	A subgroup $\Gamma$ of $\Iso(\R^n)$ is called \textit{co-compact} if $\R^n/\Gamma$ is compact. A co-compact and discrete subgroup $\Gamma$ is called a \textit{crystallographic} group. If $\Gamma$ is also torsionfree then $\Gamma$ is called a \textit{Bieberbach} group.
	 \end{defi}

For a discrete subgroup $\Gamma\subset\Iso(\R^n)$, an equivalent condition to be torsionfree is that the action of $\Gamma$ on $\R^n$ be free. A nice class of group actions are the properly discontinuous actions. In general, we say that an action of a subgroup $G$ of the group of homeomorphisms of a topological space $X$ is \textit{properly discontinuous} if	 for all $x\in X$ there exists an open neighborhood $U$ of $x$ such that $gU\cap U=\emptyset$ for all $g\neq e$. It is known that if a discrete subgroup $\Gamma$ of $\Iso(\R^n)$ acts freely on $\R^n$, then the action is properly discontinuous and $\R^n/\Gamma$ admits a structure of a differentiable manifold. As a corollary, we have the next important theorem about Bieberbach groups.

\begin{teo}\label{biebvartop}
	Let $\Gamma$ be a Bieberbach subgroup of $\Iso(\R^n)$ and $\pi:\R^n\to\R^n/\Gamma$ the canonical projection. Then $\Gamma$ acts on $\R^n$ in a properly discontinuous manner, hence $\pi$ is a covering map and $\pi_1(\R^n/\Gamma)\cong \Gamma$. Furthermore $\R^n/\Gamma$ admits a structure of a $n$-dimensional differentiable manifold such that $\pi$ is a local diffeomorphism. 
\end{teo}

Bieberbach groups, and more generally crystallographic groups, are well described by three classic theorems known as \comillas{Bieberbach's theorems}. A proof of them can be found in \cite{Ch}.

\begin{teo}[Bieberbach's First Theorem]\label{maximal}
	Let $\Gamma$ be a crystallographic subgroup of $\Iso(\R^n)$ and $\Lambda$ the subgroup of pure translations of $\Gamma$. Then $\Lambda$ is a normal free abelian subgroup of rank $n$ such that $\Gamma/\Lambda$ is a finite group. Furthermore, $\Lambda$ is maximal as a normal abelian subgroup, in the sense that it contains every other normal abelian subgroup of $\Gamma$.
	\end{teo}
	
	\medskip
	
	The group $\Gamma/\Lambda$ is known as the \textit{point group} of $\Gamma$ or the \textit{holonomy group} of $\Gamma$, in the case that $\Gamma$ is a Bieberbach subgroup of $\Iso(\R^n)$. The point group of $\Gamma$ can be identified with $r(\Gamma)\subset\ort(n)$, which is a finite group, where $r$ is the projection homomorphism given by $r(A,v)=A.$ Indeed, the kernel of $r:\Gamma\to\ort(n)$ is $\Lambda$ and thus $\Gamma/\Lambda\cong r(\Gamma)$.
	
	\medskip
	
		\begin{teo}[Bieberbach's Second Theorem]
			Let $\Gamma_1$ and $\Gamma_2$ be crystallographic subgroups of $\Iso(\R^n)$ and let $f:\Gamma_1\to\Gamma_2$ a group isomorphism. Then there exists $\alpha\in\A_n$ such that $f(\beta)=\alpha\beta\alpha^{-1}$ for all $\beta\in\Gamma_1$. In other words, every isomorphism of crystallographic groups can be realized by an affine change of coordinates.
			\end{teo} 
			
			\begin{teo}[Bieberbach's Third Theorem]  
			For each $n\in\N$, there are only finitely many isomorphism classes of crystallographic subgroups of $\Iso(\R^n)$.
				\end{teo}

%\begin{prop}[Properties of holonomy]\label{prodriem}
%	\hfill
%\begin{enumerate}
%		\item 
%	Let $F:(M,\nabla^M)\to (N,\nabla^N)$ be an affine equivalence. Then \[\hol_p(\nabla^M)\cong\hol_{F(p)}(\nabla^N),\quad\text{for all}\;p\in M.\]
	
%	\item Let $(M_1,\la\,,\ra^1)$ and $(M_2,\la\,,\ra^2)$ be Riemannian manifolds with its Levi-Civita connections $\nabla^1$ and $\nabla^2$, respectively. If we consider on $M_1\times M_2$ the product metric and the Levi-Civita connection, then \[\hol_{(p,q)}(M_1\times M_2)=\hol_p(M_1)\times\hol_q(M_2),\quad\text{for all}\;(p,q)\in M_1\times M_2.\]
%\end{enumerate}
	
%\end{prop}

\medskip

Bieberbach's theorems are fundamental tools to study compact flat manifolds. Indeed, a classical theorem asserts that a Riemannian manifold $M$ is flat if and only if $M$ is isometric to $\R^n/\Gamma$ where $\Gamma$ is a subgroup of $\Iso(\R^n)$ that acts on $\R^n$ in a properly discontinuous manner, that is to say, $\Gamma$ is a Bieberbach subgroup of $\Iso(\R^n)$. Furthermore, the holonomy group of the Riemannian manifold $\R^n/\Gamma$ can be identified with what is called the holonomy group of the Bieberbach group, that is $\hol(\R^n/\Gamma)\cong r(\Gamma)$, see for instance \cite[Example 3.1]{Ch}.

We can reinterpret Bieberbach's theorems in the context of compact flat manifolds.

\begin{teo}[Bieberbach First]
	Let $M$ a compact flat manifold. Then $M$ admits a Riemannian covering by a flat torus and the covering is a local isometry. Furthermore, the holonomy group $\hol(M)$ is finite. 
\end{teo}

\begin{teo}[Bieberbach Second]\label{Biebvar2}
	Let $M$ and $N$ be compact flat manifolds with isomorphic fundamental groups. Then there exists an affine equivalence\footnote{An \textit{affine equivalence} between two Riemannian manifolds $M$ and $N$ equipped with connections $\nabla^M$ and $\nabla^N$ respectively, is a diffeomorphism $F:M\to N$ such that $\nabla_X^M Y=\nabla^N_{dFX}(dF Y)$ for all $X,Y\in\X(M)$.} between $M$ and $N$ equipped with the corresponding Levi-Civita connections.
\end{teo}

\begin{teo}[Bieberbach Third]
	For each $n\in\N$, there are only finitely many classes of affine equivalence of compact flat manifolds of dimension $n$.
\end{teo}

\medskip

\subsection{Lie groups and flat solvmanifolds}
In the context of Lie groups, the most interesting metrics are the left invariant ones, which satisfy that the left translations are isometries. When $G$ is an abelian Lie group, it is clear that every left invariant metric is flat. An interesting fact proved by Milnor is that some left invariant metrics may be flat, even though $G$ is not abelian. Indeed, he described the structure of a Lie group admitting one of these flat metrics (see \cite[Theorem 1.5]{Mi}). The precise criterion is stated as follows.

\begin{teo}\label{algebraplana}
	A left invariant metric on a Lie group $G$ is flat if and only if the associated Lie algebra $\g$ splits as an orthogonal direct sum $\g=\b\oplus\u$, where $\b$ is an abelian subalgebra, $\u$ is an abelian ideal and the linear transformation $\ad_b$ is skew-adjoint for every $b\in\b$.
\end{teo}

We will call $(G,\la\cdot,\cdot\ra)$ a flat Lie group if $\la\cdot,\cdot\ra$ is a flat left invariant metric on the Lie group $G$ and $(\g,\la\cdot,\cdot\ra_e)$ will be called a flat Lie algebra. 

Using Milnor's characterization of flat Lie algebras, Barberis, Dotti and Fino decompose even more a flat Lie algebra \cite[Proposition 2.1]{BDF}. 

\begin{teo}\label{alglieplanas}
	Let $(\g,\la\cdot,\cdot\ra_e)$ be a flat Lie algebra. Then $\g$ splits as an orthogonal direct sum, \[\g=\z(\g)\oplus\b\oplus [\g,\g]\] where $\b$ is an abelian subalgebra, $[\g,\g]$ is abelian and the following conditions are satisfied:
	
	\begin{enumerate}
		\item $\ad:\b\to\mathfrak{so}[\g,\g]$ is injective, $\dim [\g,\g]$ is even and $\dim\b\leq\frac{\dim[\g,\g]}{2}$;
		
		\item $\nabla_X=\ad_X$ for all $X\in\b$ and $\nabla_X=0$ for all $X\in\z(\g)\oplus[\g,\g]$.
	\end{enumerate}
\end{teo}

As a consequence, since $\{\ad_X\mid X\in\b\}$ is an abelian subalgebra of $\so[\g,\g]$, it is included in a maximal abelian subalgebra and as all of these are conjugated,  there exists an orthonormal basis $\B$ of $[\g,\g]$ such that for $X\in\b$, \[[\ad_X]_\mathcal{B}=\matriz{0&a_1&&&\\-a_1&0&&&\\&&\ddots&&\\&&&0&a_m\\&&&-a_m&0},\quad\text{with}\;a_1,\,\ldots\,, a_m\in\R, \;m=\frac{\dim [\g,\g]}{2}.\]
 
Note that a flat Lie algebra $(\g,\la\cdot,\cdot\ra_e)$ is $2$-step solvable, since $[\g,\g]$ is abelian; and unimodular\footnote{A Lie algebra $\g$ is said to be \textit{unimodular} if $\tr \ad_X=0$ for all $X\in\g$.}, since $\ad_X$ is skew-adjoint for all $X\in\g$.

\bigskip

Now, we will see that we can obtain a special class of compact flat manifolds as quotients of Lie groups with a flat invariant metric. We recall first a theorem about the differentiable structure of a quotient manifold $G/\varGamma$ in the case that $\varGamma$ is a closed subgroup of the Lie group $G$.

\begin{teo}
	Let $\varGamma$ be a closed subgroup of a Lie group $G$. The left coset space $G/\varGamma$ is a topological manifold of dimension equal to $\dim G-\dim \varGamma$, and has a unique smooth structure such that the quotient map $\pi:G\to G/\varGamma$ is a smooth submersion.
\end{teo}

When $\varGamma$ is a discrete subgroup of a Lie group $G$, hence closed and $0$-dimensional, the quotient map $\pi:G\to G/\varGamma$ is a local diffeomorphism. Let us assume now that $G$ carries a left invariant Riemannian metric. Instead of working with the left coset space, we will consider the right coset space $\varGamma\backslash G$ (note that the theorem above is also valid for the right coset space), to be able to induce a Riemannian metric on $\varGamma\backslash G$ so that $\pi$ results an isometry: Given $g\in G$ and $u,v\in T_{\pi(g)}(\varGamma\backslash G)$, we define \[\la u,v\ra_{\pi(g)}=\la(d\pi)_g^{-1}u,(d\pi)_g^{-1}v\ra_g.\] The metric is well defined since $L_\gamma$ is an isometry of $G$ and $\pi\circ L_\gamma=\pi$ for all $\gamma\in\varGamma$, where $L_\gamma(g)=\gamma g$ for all $g\in G$. It is clear that $\pi$ is a local isometry. In particular, if the metric on $G$ is flat, the induced metric on $\varGamma\backslash G$.

\medskip

The most interesting case is when $\varGamma$ is a co-compact discrete subgroup of $G$, which is called a \textit{lattice} of $G$.

\begin{defi}
	A \textit{solvmanifold} $M$ is a compact quotient $\varGamma\backslash G$ where $G$ is a simply connected solvable Lie group and $\varGamma$ is a lattice of $G$. When $G$ is nilpotent, $M$ is said to be a \textit{nilmanifold}.
\end{defi}

\begin{obs}
	It is worth mentioning that in the literature, different definitions of solvmanifolds are also considered. For example, L. Auslander defines a solvmanifold as a \comillas{homogeneous space of connected solvable Lie group}, \cite[p. 398]{Aus61}, see also \cite{Aus73a,Aus73b}. %and K. Hasegawa considers solvmanifolds as compact manifolds where a solvable Lie group acts transitively. 
	We will only consider solvmanifolds as in our definition above, and as a consequence they are always compact, orientable and paralellizable.
\end{obs}

\medskip

Every simply connected solvable Lie group is diffeomorphic to $\R^n$ (see for instance \cite{Va}) for some $n\in\N$. This implies that a solvmanifold is aspherical, that is, the higher homotopy groups	vanish; and $\pi_1(\varGamma\backslash G)\cong\varGamma$. The fundamental group plays an important role in the study of solvmanifolds. Indeed, Mostow's Theorem show that the solvmanifolds are classified, up to diffeomorphism, by their fundamental groups. A proof of this can be found in \cite[Theorem 3.6]{Rag}.

\begin{teo}[Mostow]
	Let $G_1$ and $G_2$ be simply connected solvable Lie groups with $\varGamma_i$ a lattice in $G_i$ for $i=1,2$. If $\phi:\varGamma_1\to\varGamma_2$ is an isomorphism then there exists a diffeomorphism $\tilde{\phi}:G_1\to G_2$ such that $\tilde{\phi}|_{\varGamma_1}=\phi$ and $\tilde{\phi}(\gamma g)=\phi(\gamma)\tilde{\phi}(g)$ for all $\gamma\in\varGamma_1, g\in G_1$.
\end{teo}
\begin{coro}
	Two solvmanifolds with isomorphic fundamental groups are diffeo\-morphic.
\end{coro}

Concerning the construction of solvmanifolds and nilmanifolds, Malcev's theorem \cite{Ma} asserts that a simply connected nilpotent Lie group $G$ admits a lattice if and only if its Lie algebra has a basis such that the structure constants are rational numbers, so it is not so difficult to construct nilmanifolds. Unfortunately, it is not easy to construct solvmanifolds, since there is not a general method to find lattices in a given solvable Lie group. It is known (\cite{Mi}) that a necessary condition for a solvable Lie group to admit lattices is to be unimodular, but in general there are no known sufficient conditions. 
  
\bigskip

A special class of solvmanifolds that we are going to consider later are those obtained as quotients of almost abelian Lie groups, which we will call \textit{almost abelian solvmanifolds}.

\begin{defi}
A non-abelian Lie algebra $\g$ is said to be \textit{almost abelian} if $\g$ has a codimension 1 abelian ideal, that is $\g=\R x\ltimes_{\ad_x}\R^n$ for some $n\in\N$.
\end{defi}

It is clear that an almost abelian Lie algebra is solvable and it is nilpotent if and only if $\ad_x|_{\R^n}$ is nilpotent. Moreover, if we denote by $\g_A$ the almost abelian Lie algebra with $\ad_x|_{\R^n}=A$, it can be proved that $\g_A\cong\g_B$ if and only if there exist $0\neq\lambda\in\R$ and $C\in\GL(n,\R)$ such that $B=\lambda CAC^{-1}$.

The simply connected Lie group $G$ corresponding to an almost abelian Lie algebra $\g=\R x\ltimes_{\ad_x}\R^n$, is called an \textit{almost abelian} Lie group, and can be written as $G=\R\ltimes_{\varphi}\R^n$, where $\varphi(t)=\exp(t\ad_x)$.

A great advantage of almost abelian Lie groups is that there exists a criterion to determine whether they admit lattices or not (see \cite{B}).
 
 \begin{prop}\label{bock}
 	Let $G=\R\ltimes_\varphi\R^m$ be an almost abelian Lie group. Then $G$ admits a lattice $\varGamma$ if and only if there exists $t_0\neq 0$ such that $\varphi(t_0)$ is conjugate to an integer matrix invertible in $\Z$. In this case, the lattice $\varGamma$ is given by \[\varGamma=t_0\Z \ltimes_\varphi P\Z^m,\] where $P^{-1}{\varphi(t_0)}P$ is an integer matrix.
 \end{prop}
 
 \medskip
 
\section{Holonomy of flat solvmanifolds}
In this section we are going to consider solvmanifolds $\varGamma\backslash G$ equipped with a flat metric induced by a flat left invariant metric on $G$, using the characterization given by Milnor mentioned above. 

We do not assume that $G$ acts by isometries on $\varGamma\backslash G$. In fact, it is well known that if $M$ is a flat Riemannian manifold with a transitive group of isometries, then $M\cong T^m\times \R^k$ for some $k,m\in\N$ (see \cite{Wo}).	

\medskip

In general, if a simply connected Lie group $G$ admits a flat left invariant metric, then $G$ is isometric as a Riemannian manifold to $\R^n$ with the usual metric. We can identify $\varGamma$ with a subgroup of isometries of $\R^n$ as follows:

\medskip

Let $\psi:G\to\R^n$ an isometry. Define a map $i:G\to \Iso(\R^n)$ by $i(g)=\tau_g$, where $\tau_g=\psi\circ L_g\circ \psi^{-1}$. It is clear that $i$ is an injective group homomorphism, and therefore we can identify $\varGamma$ with its image and consider $ \varGamma\subset\Iso(\R^n)$.
 
The solvmanifold $\varGamma\backslash G$ is isometric to $\R^n/\varGamma$ by considering the map $\phi([g])=[\psi(g)]$. Observe that $\phi$ is well defined since \[\phi([\gamma g])=[\psi(\gamma g)]=[\tau_\gamma(\psi(g))]=[\psi(g)]=\phi([g]).\] Furthermore it is injective, since if $[\psi(g)]=[\psi(h)]$ then $\psi(g)=\tau_\gamma (\psi(h))$, which implies $\psi(g)=\psi(\gamma h)$. Thus, $g=\gamma h$, or equivalently, $[g]=[h]$. The surjectivity is clear, and as $\phi$ is a composition of local isometries and bijective, $\phi$ results an isometry. Therefore $\varGamma$ is a Bieberbach subgroup of $\Iso(\R^n)$ and 
$\hol(\R^n/\varGamma)\cong r(\varGamma)$. We recall that $r(\varGamma)$ is identified with the finite group $\varGamma/\Lambda$ where $\Lambda$ is the maximal normal abelian subgroup of $\varGamma$. In conclusion,
\begin{equation}\label{eqhol}
\hol(\varGamma\backslash G)\cong \varGamma/\Lambda.
\end{equation}

\begin{obs}\label{primitivo}
	A compact flat manifold $M=\R^n/\Gamma$ is called \textit{primitive} when its first Betti number $\beta_1(M)$ is zero. It is known that $M$ is primitive if and only if the center of $\Gamma$ is trivial, see for instance \cite{Sz}. Any solvmanifold has $\beta_1\geq 1$ due to a result of Hattori \cite{Hat}. Therefore, any flat solvmanifold $\varGamma\backslash G$ is not primitive and $\varGamma$ has non trivial center.
\end{obs}

\medskip

Combining Milnor's result (Theorem \ref{algebraplana}) with equation \eqref{eqhol} and the maximality of the normal abelian subgroup $\Lambda$, we give an elementary proof of the result in \cite{Aus}.

\begin{teo}\label{holabeliana}
	The holonomy group of every flat solvmanifold is abelian.
\end{teo}

\begin{proof}[Proof:\nopunct]
	Let $G$ be a simply connected Lie group with a flat left invariant metric and let $\varGamma$ be a lattice in $G$. If $\g=\lie(G)$, it follows from Theorem \ref{algebraplana} that we can write $\g=\b\oplus\u$ where $\b$ is an abelian subalgebra, $\u$ is an abelian ideal and $\ad_x:\u\to\u$ is skew-adjoint for every $x\in\b$. We can rewrite $\g$ as $\g=\R^k\ltimes_{\ad}\R^\l$ where $k=\dim\b$ and $\l=\dim\u\geq 2$. Therefore $G$ can be written as the semidirect product $G=\R^k\ltimes_{\varphi}\R^\l$ for some homomorphism $\varphi:\R^k\to\Aut(\R^\l)$. 
	
	Since $\varGamma\subset \R^k\ltimes_\varphi\R^\l$ is a Bieberbach group, there exists a maximal normal abelian subgroup $\Lambda$ of finite index in $\varGamma$. The subgroup $\varGamma\cap\R^\l$ is a normal abelian subgroup of $\varGamma$. Since $\Lambda$ is maximal (Proposition \ref{maximal}) we have that $\varGamma\cap\R^\l\subset \Lambda$. We show next that $\varGamma/(\varGamma\cap\R^\l)$ is abelian. Indeed, if we consider the composition \[\varGamma\xhookrightarrow[]{\iota}\R^k\ltimes\R^\l\stackrel{\pi}{\longrightarrow}(\R^k\ltimes\R^\l)/\R^\l,\]  we have \[\varGamma/(\varGamma\cap\R^\l) \cong \im(\pi\circ\iota)<(\R^k\ltimes\R^\l)/\R^\l\cong\R^k,\] then the group $\varGamma/(\varGamma\cap\R^\l)$, being isomorphic to a subgroup of an abelian group, is abelian.
	
	Now, the natural map $\varGamma/(\varGamma\cap\R^\l)\to\varGamma/\Lambda$ is an epimorphism and  $\varGamma/(\varGamma\cap\R^\l)$ is abelian, therefore $\varGamma/\Lambda$ is abelian. It follows from \eqref{eqhol} that the holonomy group is abelian.
\end{proof}

\smallskip

Our next goal is to prove the converse of Theorem \ref{holabeliana}, namely, that every finite abelian group can be realized as the holonomy group of a flat solvmanifold. To do this, we focus on studying flat almost abelian Lie groups, in which we can construct lattices using Proposition \ref{bock}. First, we characterize the almost abelian Lie algebras in terms of the decomposition of a flat Lie algebra as in Theorem \ref{alglieplanas}.

\begin{teo}\label{casiabelianaplana}
	Let $(\g,\la\cdot,\cdot\ra)$ be a flat Lie algebra and let \begin{equation}\label{eq12}
	\g=\z(\g)\oplus\b\oplus [\g,\g]
	\end{equation} be its decomposition as an orthogonal direct sum as in Theorem \ref{alglieplanas}. Then $\g$ is almost abelian if and only if $\dim\b=1$. 
\end{teo}
\begin{proof}[Proof:\nopunct] Let us assume first that $\g$ is almost abelian and let $\u\cong\R^n$ be an abelian ideal of codimension 1. Let $x\neq 0$ such that $x\in\u^\perp$, then $\g$ can be written as $\g=\R x\ltimes_{\ad_x} \u$, with $\R x\perp \u$. We will prove that $\b=\R x$ in several steps:
	
	$\ri$ First, we prove that $[\g,\g]\oplus\z(\g)\subset \u$. Indeed, $[\g,\g]\subset\u$ given that $\u$ is an ideal and also we have that $\z(\g)\subset\u$ because if $z\in\z(\g)$ decomposes as $z=ax+u$ for $a\in\R, a\neq 0,$ and $u\in\u$, then \[0=[z,x]=[u,x].\] Since $\u$ is abelian, it follows that $u\in\z(\g)$ and thus, $x\in\z(\g)$ which contradicts the fact that $\g$ is not abelian. Therefore, $a=0$ and $z\in\u$ so that \[\z(\g)\oplus [\g,\g]\subset \u.\] 
	
	$\rii$ We show next that $x\in\b$. We write $x=z+b+y$ according to \eqref{eq12} with $z\in\z(\g), b\in\b$ and $y\in[\g,\g]$. Due to $\ri$ we get $0=\la x,z\ra=\la x,y\ra$. Since $\la x,z\ra=\norm{z}^2+\la b,z\ra+\la y,z\ra=\norm{z}^2$ and similarly $\la x,y\ra=\norm{y}^2$,  it follows that $0=\norm{z}^2=\norm{y}^2$ which implies $x\in\b$. 
	
	\medskip
	
	$\riii$ Finally, we will prove that $\b=\R x$. If we consider an element $v\in\b, v\perp x$, then $v\in \b\cap\u$. As $x\in\b$ and $\u$ is abelian it follows that $v\in\z(\g)$ and thus $v\in\z(\g)\cap\b=\{0\}$. This proves that $\b=\R x$.
	
	\medskip
	
	Conversely, if $\dim\b=1$ then the ideal $\z(\g)\oplus [\g,\g]$ is abelian and has codimension 1. Therefore, $\g$ is almost abelian.
\end{proof}

As a consequence of this theorem we have that if $\g=\R x\ltimes_{\ad_x}\R^m$ is a flat Lie algebra, then $\ad_x:\R^m\to\R^m$ is skew-adjoint and therefore \[\R^m=\Ker(\ad_x)\oplus \im(\ad_x),\] where $\Ker(\ad_x)=\z(\g)$,  $\im(\ad_x)=[\g,\g]$ and $\dim[\g,\g]$ is even. Therefore there exist bases  $\mathcal{B}=\{e_1,f_1\,\ldots\,,e_n,f_n\}$ of $[\g,\g]$ and  $\mathcal{B}'$ of $\z(\g)$ such that \[ [\ad_x]=\matriz{0_{s\times s} &&&&&\\&0&a_1&&&\\&-a_1&0&&&\\&&&\ddots&&\\&&&&0&a_n\\&&&&-a_n&0},\] where $s=\dim\z(\g)$ and $a_1,\ldots,a_n\in\R-\{0\}$.

\medskip

The Lie algebra $\g=\R x\ltimes_{\ad_x}\R^{s+2n}$ is completely determined by $s,a_1,\,\ldots\,,a_n$. Moreover, by reordering the basis we can assume that $a_1\geq\cdots\geq a_n>0$.

\medskip

The simply connected Lie group associated to $\g=\R x\ltimes_{\ad_x} \R^{s+2n}$ is \[G=\R\ltimes_{\varphi}\R^{s+2n},\] where
\[\varphi(t)=\exp(t\ad_x)=\matriz{\Id_{s\times s}&0\\0&\tilde{\varphi}(t)}\] with
\begin{equation}\label{eq11}
\tilde{\varphi}(t)=\matriz{\cos(a_1 t)&\sin(a_1 t)&&&\\-\sin(a_1 t)&\cos(a_1 t)&&&\\&&\ddots&&\\&&&\cos(a_n t)&\sin(a_n t)\\&&&-\sin(a_n t)&\cos(a_n t)}
\end{equation}
and $a_1\geq\cdots\geq a_n>0$.

Now we want to construct lattices in these almost abelian Lie groups. Concerning Proposition \ref{bock}, it can be guessed that the values of $t_0$ and $a_1,\ldots,a_n$ so that $\tilde{\varphi}(t)$ is conjugate to an integer matrix are closely related to the multiples of $\pi$. We prove next that this assertion is true, proving that the order of the matrix $\tilde{\varphi}(t)$ is finite, using a theorem of Kronecker about roots over $\C$ of monic integer polynomials. A proof can be found in \cite{Gre}.

\begin{teo}[Kronecker]
	Let $\lambda\neq 0$ be a root of a monic polynomial $f(z)$ with integer coefficients. If all the roots of $f(z)$ are in the unit disc $\{z\in\C\mid |z|\leq 1\}$ then $\lambda$ is a root of unity.
\end{teo}

\begin{teo}\label{conjenteraordfinito}
	Let $t_0\neq 0$ such that the matrix $\varphi(t_0)=\matriz{\Id_{s\times s}&0\\0&\tilde{\varphi}(t_0)}$ as in \eqref{eq11} is conjugate to an integer matrix. Then $\ord(\varphi(t_0))<\infty$.
\end{teo}

\begin{proof}[Proof:\nopunct]
	It suffices to show that the matrix $\tilde{\varphi}(t_0)$ has finite order. We write \[\tilde{\varphi}(t_0)=\matriz{\theta(a_1 t_0)&&&\\&&\ddots&\\&&&\theta(a_n t_0)}\quad\text{where}\; \theta(t)=\matriz{\cos(t)&\sin(t)\\-\sin(t)&\cos(t)}.\] If there exists $P\in\GL(s+2n,\R)$ such that $P^{-1}\varphi(t_0)P$ is an integer matrix, then the coefficients of the characteristic and minimal polynomials of $\varphi(t_0)$ are integers. It is straightforward to prove that the eigenvalues of $\varphi(t_0)$ are \[1,\ldots,1,\cos(a_1 t_0)\pm i \sin(a_1 t_0),\,\ldots\,,\cos(a_n t_0)\pm i\sin(a_n t_0),\] which all lie in the closed unit disc. 
	
	Therefore, we apply Kronecker's Theorem and we conclude that all the eigenvalues of $\varphi(t_0)$ are roots of unity. In particular, given $j\in\N, 1\leq j\leq n$, there exists $k_j\in\N$ such that ${(\cos(a_j t_0)+ i \sin(a_j t_0))}^{k_j}=1$, which implies $\theta(a_j t_0)^{k_j}=\Id$. Thus, \[\ord(\tilde{\varphi}(t_0))=\lcm(k_1,\,\ldots\,,k_n).\] Hence, $\ord\tilde{\varphi}(t_0)<\infty$.
\end{proof}

\begin{obs}
	Note that $\ord(\theta(t))<\infty$ if and only if $t\in \pi\Q$. Therefore, a necessary condition for $\varphi(t_0)$ to be conjugate to an integer matrix is that $a_i t_0\in \pi\Q$ for every $1\leq i\leq n$. However, we will see later that this is not a sufficient condition. A sufficient condition is that the complex eigenvalues of $\tilde{\varphi}(t_0)$ are all different and the coefficients of its characteristic polynomial are integers. In this case, it is well known that $\tilde{\varphi}(t_0)$ is conjugate over $\C$ to the companion matrix over $\C$ of its characteristic polynomial, which is an integer matrix. However, if two real matrices are conjugate over $\C$ then they are conjugate over $\R$.
\end{obs}

\medskip

Once we know how to construct lattices in flat almost abelian Lie groups, we need to compute the holonomy group of the corresponding solvmanifolds, what we do in the next result. In particular, we show that the holonomy group of a flat almost abelian solvmanifold is cyclic.

\begin{teo}\label{holonomia_lcm}
	Let $G=\R\ltimes_\varphi\R^{s+2n}$ be a flat almost abelian Lie group and let $t_0\neq 0$ such that the matrix $\varphi(t_0)=\matriz{\Id_{s\times s}&0\\0&\tilde{\varphi}(t_0)}$ is conjugate to an integer
		 matrix through a matrix $P\in\GL(s+2n,\R)$, where \[\tilde{\varphi}(t_0)=\matriz{\theta(a_1 t_0)&&\\&\ddots&\\&&\theta(a_n t_0)}\quad\text{with}\;\theta(t)=\matriz{\cos(t)&\sin(t)\\-\sin(t)&\cos(t)}.\] If $\varGamma$ is the lattice given by $\varGamma=t_0\Z\ltimes_\varphi P\Z^{s+2n}$ and \[d=\lcm\left(\ord(\theta(a_1 t_0)),\,\ldots\,,\ord(\theta(a_n t_0))\right),\]   then\[\hol(\varGamma\backslash G)=\begin{cases}
	\Z_d&\text{if}\;d>1,\\
	\{e\}&\text{if}\;d=1.
	\end{cases}\]
\end{teo}
%Cambiar Lambda por otra cosa (un sigma), y H por lambda
\begin{proof}[Proof:\nopunct]
	If $d=1$, then $\varphi(t_0)=\Id$, hence $G$ is abelian and thus the solvmanifold $\varGamma\backslash G$ is a torus, which has trivial holonomy group. If $d>1$, we claim that the maximal normal abelian subgroup is $\Sigma=d t_0 \Z\ltimes_\varphi P\Z^{s+2n}$. Indeed, it is clear that $dt_0\Z\ltimes_\varphi P\Z^{s+2n}$ is normal in $\varGamma$ and abelian. Furthermore, we show next that $\Sigma$ is maximal: let $\Lambda$ be the maximal normal abelian subgroup of $\varGamma$. According to Proposition \ref{maximal}, $\Sigma\subset\Lambda$, and we will prove that $\Lambda\subset\Sigma$. Let $\gamma_{\lambda}=(t_0 k,Pm) \in \Lambda$ and $\gamma_\sigma=(d t_0 \l, Pr)\in\Sigma$.  It suffices to show that $k=d k_0$ for some $k_0\in\Z$. Due to the fact that $\gamma_\lambda$ and $\gamma_\sigma$ commute, we deduce, looking at the last coordinates, that
	\[Pm+\varphi(t_0)^k Pr=Pr+Pm,\] which implies $\varphi(t_0)^k Pr=Pr$. Then $P^{-1}\varphi(t_0)^k P$ is a matrix that fixes $r$ for all $r\in\Z^{s+2n}$. If we let $r$ vary through the vectors of the canonical basis, it follows that $P^{-1}\varphi(t_0)^kP=\Id$ and thus $\varphi(t_0)^k=\Id$. Now, $d=\ord(\tilde{\varphi}(t_0))$, so $k=d k_0$ for some $k_0\in\Z$. Therefore $\Lambda\subset \Sigma$ and this concludes the proof of the statement.
	
	Finally, let $\phi:\varGamma\to\Z_d$ be the epimorphism defined by \[\phi(t_0 k,Pm)=[k]_d.\]
	It is clear that $(t_0 k,Pm)\in\Ker(\phi)\iff k=d k_0\iff (t_0 k,Pm)\in\Lambda$. The formula $\hol(\varGamma\backslash G)\cong \varGamma/\Lambda$ implies $\hol(\varGamma\backslash G)\cong \Z_d$. 
\end{proof}
\begin{obs}
	Note that if $\varphi(t_0)$ is conjugate to an integer matrix $E$, then $\det(E)=1$, so $E$ is invertible over $\Z$.
\end{obs}

\bigskip

Finally, to prove the converse of Theorem \ref{holabeliana}, we consider the case when $(\g,\la\cdot,\cdot\ra)$ is a flat Lie algebra with $\dim\g=2n+1$ and $\dim[\g,\g]=2n$ for some $n\in\N$. Hence, $\dim\b=1$ so $\g$ is almost abelian and has trivial center, or equivalently, $\ad_x:\R^{2n}\to\R^{2n}$ is invertible. We can write  \[[\ad_x]_{\mathcal{B}}=\matriz{0&a_1&&\\-a_1&0&&&\\&&\ddots&&\\&&&0&a_n\\&&&-a_n&0},\] for some orthonormal basis $\mathcal{B}$ of $\R^{2n}$ and $a_1\geq\cdots\geq a_n>0$. The corresponding simply connected Lie group is $G=\R\ltimes_\varphi \R^{2n}$ where \begin{equation}\label{eq20}
\varphi(t)=\matriz{\cos(a_1 t)&\sin(a_1 t)&&&\\-\sin(a_1 t)&\cos(a_1 t)&&&\\&&\ddots&&\\&&&\cos(a_n t)&\sin(a_n t)\\&&&-\sin(a_n t)&\cos(a_n t)}.
\end{equation} 

\bigskip

\begin{teo}\label{holpotprimos} Let $p$ be an odd prime and let $k\in\N$. Then there exist flat solv\-manifolds of dimension $(p-1)p^{k-1}+1$ with holonomy group $\Z_{p^k}$ and $\Z_{2p^k}$, respectively. If $p=2$ the assertion is valid for $k\geq 2$.
\end{teo}

\begin{proof}[Proof:\nopunct]
	We will start by constructing a flat solvmanifold of dimension $2^{k-1}+1$ with holonomy group $\Z_{2^k}$, for $k\geq 2$. Let $n=2^{k-2}$, we choose $t_0=\frac{\pi}{2^{k-1}}$ and $a_\l=2\l-1$ for $1\leq \l\leq n$, so that \[a_1 t_0=\frac{\pi}{2^{k-1}},\quad \ldots\quad,\quad a_n t_0=\frac{(2^{k-1}-1)\pi}{2^{k-1}}.\]
	The eigenvalues of $\varphi(t_0)$ as in \eqref{eq20} for this choice of $t_0$ and $a_\l$ are \[\exp\left({\pm \frac{j\pi i}{2^{k-1}}}\right),\;\; 1\leq j\leq 2^{k-1}-1,\;\; j\,\text{odd}.\]
	Note that \[\exp\left(-\frac{j\pi i}{2^{k-1}}\right)=\exp\big((2^k-j)\frac{2\pi i}{2^k}\big),\] and if $1\leq j\leq 2^{k-1}-1$ then $2^{k-1}+1\leq 2^k-j\leq 2^k-1$ . In conclusion, the set of eigenvalues of $\varphi(t_0)$ is $\{\exp(\frac{2\pi i}{2^k}j)\mid 1\leq j\leq 2^k, \, \gcd(j,2^k)=1\}$ which is exactly the set of primitive $2^k$-th roots of unity. Therefore, the characteristic polynomial of $\varphi(t_0)$ is the $2^k$-th cyclotomic polynomial $\Phi_{2^k}$ which has integer coefficients (it is well known that $\Phi_n$ has integer coefficients for all $n\in\N$). Since all the eigenvalues of $\varphi(t_0)$ are different, $\varphi(t_0)$ is conjugate to the companion matrix of $\Phi_{2^k}$ which has integer coefficients. Then, by Proposition \ref{bock}, the associated Lie group $G=\R\ltimes_{\varphi}\R^{2^{k-1}}$ (for this choice of $t_0$ and $a_\l$) admits a lattice \[\varGamma_{2^{k-1}}=\frac{\pi}{2^{k-1}}\Z\ltimes_{\varphi}P\Z^{2^{k-1}}, \quad\text{for some matrix}\;P\in\GL(2^{k-1},\R).\]
	It follows from Theorem \ref{holonomia_lcm} that the holonomy group of the solvmanifold $\varGamma_{2^{k-1}}\backslash G$ is $\Z_d$ where \[d=\lcm\left(\ord\left(\theta\left(\frac{\pi}{2^{k-1}}\right)\right),\; \ord\left(\theta\left(\frac{3\pi}{2^{k-1}}\right)\right),\,\ldots\,,\;\ord\left(\theta\left(\frac{(2^{k-1}-1)\pi}{2^{k-1}}\right)\right)\right)=2^k.\]
	
	\bigskip
	
	We will continue by exhibiting a solvmanifold of dimension $(p-1)p^{k-1}+1$ with holonomy group $\Z_{p^k}$ for $p$ an odd prime and $k\geq 1$. Let $n=\frac{(p-1)p^{k-1}}{2}$ and $t_0=\frac{2\pi}{p^k}$. To choose the parameters $\{a_\l\}_{\l=1}^n$, note that there are exactly $n$ integers $\l$ such that $1\leq \l\leq \frac{p^k-1}{2}$ and $\gcd(\l,p^k)=1$ since the numbers that are not relatively prime to $p^k$ in this range are $pm$ for $1\leq m\leq \frac{p^{k-1}-1}{2}$. We choose $\{a_\l\}_{\l=1}^n$ to be these $n$ numbers and thus
	the eigenvalues of $\varphi(t_0)$ are \[\exp\left(\pm \frac{2\pi i}{p^k}\l\right),\quad 1\leq \l\leq \frac{p^{k}-1}{2},\; \gcd(\l,p^k)=1.\] Again, note that $\exp(-\frac{2\pi i}{p^k}\l)=\exp(\frac{2\pi i p^k-2\pi i \l}{p^k})=\exp((p^k-\l)\frac{2\pi i}{p^k})$. Furthermore, if $\gcd(\l,p^k)=1$ then  $\gcd(p^k-\l,p^k)=1$ and if $1\leq \l \leq \frac{p^k-1}{2}$ then $\frac{p^k+1}{2}\leq p^k-\l\leq p^k-1$. Therefore, the eigenvalues of $\varphi(t_0)$ are exactly the primitive $p^k$-th roots of unity. Thus, the characteristic polynomial of $\varphi(t_0)$ is the $p^k$-th cyclotomic polynomial $\Phi_{p^k}$ and analogously to the previous case,  
	%since all the eigenvalues of $\varphi(t_0)$ are different, $\varphi(t_0)$ is conjugate to the companion matrix of $\Phi_{p^k}(x)$, which is integer. Then by Proposition \ref{bock} 
	the Lie group $G=\R\ltimes_\varphi \R^{(p-1)p^{k-1}}$ admits a lattice $\varGamma_{p^k}$ of the form \[\varGamma_{p^k}=\frac{2\pi}{p^k}\Z\ltimes_\varphi P\Z^{(p-1)p^{k-1}},\quad\text{for some matrix}\;P\in\GL((p-1)p^{k-1},\R).\]
	It follows from Theorem \ref{holonomia_lcm} that the holonomy group of the solvmanifold $\varGamma_{p^k}\backslash G$ is $\Z_d$ where $d=\lcm(\ord(\theta(\frac{2\pi}{p^k}j)))=p^k$.
	
	\medskip
	
	Finally, we will exhibit a solvmanifold of dimension $(p-1)p^{k-1}+1$ with holonomy group $\Z_{2p^{k}}$, for $k\geq 1$. Let $n=\frac{(p-1)p^{k-1}}{2}$ and $t_0=\frac{\pi}{p^k}$. To choose the parameters $\{a_\l\}_{\l=1}^n$, note that there are $n$ odd integers $\l$ such that $1\leq \l\leq p^k-2$ and $\gcd(\l,2p^k)=1$ since if $\l$ is odd and not coprime with $2p^k$ then $\l=mp$ for $1\leq m\leq p^{k-1}-2$. We choose $\{a_\l\}_{\l=1}^n$ to be these $n$ odd integers and thus the eigenvalues of $\varphi(t_0)$ are

\[\exp\left(\pm \frac{\pi i}{p^k}\l\right),\quad 1\leq \l\leq p^k-2,\; \gcd(\l,2p^k)=1,\] or equivalently 

\[\exp\left(\frac{2\pi i}{2p^k}j\right),\quad 1\leq j\leq 2p^{k}-1,\; \gcd(j,2p^k)=1,\]  which is the set of primitive $(2p^k)$-th roots of unity. Thus, the characteristic polynomial of $\varphi(t_0)$ is the $(2p^k)$-th cyclotomic polynomial $\Phi_{2p^k}$. Similarly to the previous cases, by Proposition \ref{bock}, the Lie group $G=\R\ltimes_\varphi\R^{(p-1)p^{k-1}}$  admits a lattice $\varGamma_{2p^k}$ of the form \[\varGamma_{2p^k}=\frac{\pi}{p^k}\Z\ltimes_\varphi P\Z^{(p-1)p^{k-1}},\quad\text{for some matrix}\; P\in\GL((p-1)p^{k-1},\R).\]
	It follows from Theorem \ref{holonomia_lcm} that the holonomy group of the solvmanifold $\varGamma_{2p^k}\backslash G$ is $\Z_d$ where $d=\lcm(\ord(\theta(\frac{j\pi}{p^k})))=2p^k$.
\end{proof}

\medskip

Finally, we prove the converse of Theorem \ref{holabeliana}.

\begin{teo}\label{holabelianofinito}
	Let $A$ be a finite abelian group. Then there exists a flat solvmanifold with holonomy group $A$.
\end{teo}
\begin{proof}[Proof:\nopunct]
	We may decompose \[A\cong\Z_{p_1^{k_1}}\oplus\cdots\oplus\Z_{p_m^{k_m}},\] according to the structure theorem for finitely generated abelian groups, where $p_1,\,\ldots\,,p_m$ are primes (not necessarily distinct) and $k_1,\,\ldots\,,k_m\in\N$. In Theorem \ref{holpotprimos} we have constructed a solvmanifold $\varGamma_i\backslash G_i$ with holonomy group $\Z_{p_i^{k_i}}$ for all $1\leq i\leq m$. Let $G$ be the Lie group $G=G_1\times\cdots\times G_m$ and $\varGamma$ the lattice of $G$ given by the product of the lattices $\varGamma_i$, $\varGamma=\varGamma_1\times\cdots\times\varGamma_m.$ We consider the solvmanifold \[\varGamma\backslash G\cong \varGamma_1\backslash G_1\times\cdots\times\varGamma_m\backslash G_m\] equipped with the product metric, which is flat. As the holonomy group of a Riemannian product is the product of the holonomy groups, it follows that \[\hol(\varGamma\backslash G)=\hol(\varGamma_1\backslash G_1)\times\cdots\times\hol(\varGamma_m\backslash G_m)\cong \Z_{p_1^{k_1}}\oplus\cdots\oplus\Z_{p_m^{k_m}}\cong A.\]
\end{proof}

\begin{coro}\label{kähler}
	Given a finite abelian group $A$ there exists a Kähler flat solv\- manifold with holonomy group $A$. 
\end{coro}
\begin{proof}[Proof:\nopunct]
	Let $A$ be a finite abelian group. By Theorem \ref{holabelianofinito} there exists a flat solvmanifold $\varGamma\backslash G$ of dimension $n$ with holonomy group $A$. In \cite{BDF}, it is proved that every flat solvmanifold of even dimension admits an invariant complex structure compatible with the metric which is Kähler. Therefore, if $n$ is even, there is nothing to prove, and if $n$ is odd, it suffices to take the solvmanifold obtained as the quotient of the Lie group $\R\times G$ by the lattice $\Z\times \varGamma$ equipped with the product metric, which is flat. Then \[\hol(\Z\times\varGamma\backslash \R\times G)\cong \hol(S^1)\times\hol(\varGamma\backslash G)\cong \{e\}\times A\cong A.\]
\end{proof}

\section{Minimal dimension}

The dimension of the solvmanifold constructed in Theorem \ref{holabelianofinito} might not be the minimal one. It is well known that every finite group $G$ is the holonomy group of some compact flat manifold  \cite{AusK}. It is interesting to find the minimal dimension of a compact flat manifold with holonomy group $G$. In \cite{Ch}, L. Charlap classified the compact flat manifolds with holonomy group $\Z_p$, called $\Z_p$-manifolds, and he proved that the minimal dimension of such a manifold is $p$. More generally, Hiller \cite{Hi} proved that the minimal dimension, denoted $d(n)$, of a compact flat manifold with holonomy group $\Z_n$ is $\Phi(n)+1$, where $\Phi$ is an \comillas{additive} version of the Euler totient function, $\varphi$.

\medskip

\begin{defi}
	The function $\Phi$ is defined by the two following rules:
	\begin{enumerate}
		\item If $p$ is prime, $\Phi(p^k)=\varphi(p^k)$, for all $k\in\N$.
		
		\smallskip
		
		\item If $m$ and $n$ are coprime, then $\Phi(mn)=\Phi(m)+\Phi(n)$, unless $m=2$ and $n$ is an odd number, $n\geq 3$, in which case $\Phi(2n)=\Phi(n)$.
	\end{enumerate}
\end{defi} 
Note that $\Phi(1)=0$, $\Phi(2)=1$ and $\Phi(n)$ is even for all $n\geq 3$.

\medskip

Consider now the problem of finding the minimal dimension of a flat solvmanifold with holonomy group $\Z_n$, which we will denote $d_S(n)$. At first glance, we know that $d(n)\leq d_S(n)$ for all $n$. If we look at $n=2$ the inequality is strict. Indeed, $d(2)=2$ because the Klein Bottle is a compact flat manifold of dimension 2 with holonomy group $\Z_2$, and $d_S(2)=3$ given that every flat solvmanifold of dimension 2 is diffeomorphic to a torus, and the \comillas{dicosm} (see next section) is a flat solvmanifold of dimension 3 with holonomy group $\Z_2$. Fortunately, this is the only case of the strict inequality, and that is what we are going to prove immediately, namely, \[d(n)=d_S(n), \;\text{for all}\;n\geq 3.\] The main idea to obtain this result is to reduce the dimension of the solvmanifold in Theorem \ref{holabelianofinito}. 

\medskip

First we prove an auxiliary result in which we exhibit a way to construct a flat solvmanifold with holonomy group $\Z_{\lcm(m,n)}$ beginning with flat solvmanifolds with holonomy group $\Z_m$ and $\Z_n$, respectively.
\smallskip

\begin{prop}
	Let $G_a=\R\ltimes_{\varphi_a}\R^{2 d_1}$ and $G_b=\R\ltimes_{\varphi_b}\R^{2 d_2}$ be flat almost abelian Lie groups. If there exist $t_0,\,t_1\in\R-\{0\}$ such that $\varphi_a(t_0)$ and $\varphi_b(t_1)$ are conjugate to integer matrices with $\ord(\varphi_a(t_0))=m$ and $\ord(\varphi_b(t_1))=n$, then there exists a flat almost abelian solvmanifold of dimension $2d_1+2d_2+1=\dim G_a+\dim G_b-1$ with holonomy group $\Z_{\lcm(m,n)}$.
\end{prop} 

\begin{proof}[Proof:\nopunct]
	Let $\varphi_a(t)=\matriz{\theta(a_1 t)&&\\&\ddots&\\&&\theta(a_{d_1} t)}$,  $\varphi_b(t)=\matriz{\theta(b_1 t)&&\\&\ddots&\\&&\theta(b_{d_2}t)}$ and $G_d$ the Lie group $G_d=\R\ltimes_{\varphi_d}\R^{2d_1+2d_2}$ where \[\varphi_d(t)=\matriz{\varphi_a(t)&\\&\varphi_c(t)},\quad\text{with}\; c_jt_0=b_j t_1,\,\text{for}\;1\leq j\leq d_2.\] Let $P_1, P_2$ invertible matrices such that $P_1^{-1}\varphi_a(t_0)P_1=E_1$ and $P_2^{-1}\varphi_b(t_1)P_2=E_2$, with $E_1$ and $E_2$ invertible integer matrices. Then $P=\matriz{P_1&\\&P_2}$ conjugates the matrix $\varphi_d(t_0)$ to the integer matrix $\matriz{E_1&\\&E_2}$. It follows from Proposition \ref{bock} that the Lie group $G_d$ admits a lattice $\varGamma_d$, and by Theorem \ref{holonomia_lcm} the corresponding solvmanifold has holonomy group $\Z_d$ with \[d=\ord(\varphi_d(t_0))=\lcm(\ord(\varphi_a(t_0)),\ord(\varphi_c(t_1))=\lcm(m,n).\]
\end{proof}

We can generalize this result inductively and thus we obtain

\smallskip

\begin{coro}\label{construccionsuma}
	Let $G_1=\R\ltimes_{\varphi_{a_1}}\R^{2d_i}, \ldots, G_n=\R\ltimes_{\varphi_{a_n}}\R^{2d_n}$ be flat almost abelian Lie groups. If there exist $t_1, \ldots, t_n\in\R-\{0\}$ such that $\varphi_{a_i}(t_i)$ is conjugate to an invertible integer matrix $E_i$ for every $1\leq i\leq n$ and $\ord(\varphi_{a_i}(t_i))=m_i$ then there exists a flat almost abelian solvmanifold of dimension \[\sum_{i=1}^n (\dim G_i-1)+1=1+2\sum_{i=1}^n d_i\] with holonomy group $\Z_d$ with $d=\lcm(m_1,\,\ldots\,,m_n)$.
\end{coro}

Next we prove the main result of this section.

\begin{teo}\label{mindim}
	Let $n\geq 3$, then $d(n)=d_S(n)$.
\end{teo}
\begin{proof}[Proof:\nopunct]
	First we prove the statement in the particular case when $n$ is a power of a prime. If $n=p^k$ with $p$ prime and $k\in\N$ ($k\geq 2$ if $p=2$), by Theorem \ref{holpotprimos} there exists a flat almost abelian solvmanifold of dimension $(p-1)p^{k-1}+1$ with holonomy group $\Z_{p^k}$, which implies \[d_S(p^k)\leq (p-1)p^{k-1}+1=\varphi(p^k)+1=\Phi(p^k)+1=d(p^k).\] Therefore, $d_S(p^k)=d(p^k)$.
	
	\medskip
	According to the definition of $\Phi$ we will have to consider separately the case when $n$ is $2q$ with $q$ odd. We begin by considering the other cases.
	
	First let $n=2^{k_0} p_1^{k_1}\cdots p_m^{k_m}$ with $k_0\geq 2$ and $p_i$ different odd primes with $k_i\in\N$. Using Theorem \ref{holpotprimos}, we deduce that there exist $t_0,t_1,\ldots,t_m\in\R-\{0\}$ such that $\varphi_{a_i}(t_i)$ (for some choice of $a_i$ depending on $t_i$) are conjugate to integer matrices $E_i$ and $\ord(\varphi_{a_i}(t_i))=p_i^{k_i}$ for all $0\leq i\leq m$, where $p_0=2$. Considering the almost abelian Lie groups \[G_0=\R\ltimes_{\varphi_{a_0}}\R^{\Phi(2^k)}, G_1=\R\ltimes_{\varphi_{a_1}}\R^{\Phi(p_1^{k_1})},\;\ldots\;,\;G_m=\R\ltimes_{\varphi_{a_m}}\R^{\Phi(p_m^{k_m})},\] it follows from Corollary \ref{construccionsuma} that there exists a flat solvmanifold of dimension $\Phi(2^{k_0})+\Phi(p_1^{k_1})+\cdots+\Phi(p_m^{k_m})+1$ with holonomy group $\Z_n$. In consequence, 
	\begin{align*}
	d_S(n)&\leq\Phi(2^{k_0})+\Phi(p_1^{k_1})+\cdots+\Phi(p_m^{k_m})+1\\
	&=\Phi(2^{k_0} p_1^{k_1}\cdots p_m^{k_m})+1\\
	&=d(n).
	\end{align*}
	Therefore $d_S(n)=d(n)$. Clearly, if $k_0=0$ the proof follows analogously. 
	
	Finally, if $n=2q$ with $q$ odd, we can write $n=2 p_1^{k_1}\cdots p_m^{k_m}$ with $p_i$ different odd primes and $k_i\in\N$. According to Theorem \ref{holpotprimos}, there exist $t_1,\ldots,t_m\in\R-\{0\}$ such that $\varphi_{a_i}(t_i)$ (for some choice of $a_i$ depending on $t_i$) are conjugate to integer matrices $E_i$ for $1\leq i\leq m$ and, furthermore $\ord(\varphi_{a_1}(t_1))=2 p_1^{k_1}$ and $\ord(\varphi_{a_i}(t_i))=p_i^{k_i}$ for $2\leq i\leq m$. Thus, if we consider the almost abelian Lie groups $G_1=\R\ltimes_{\varphi_{a_1}}\R^{\Phi(2 p_1^{k_1})}$ and $G_i=\R\ltimes_{\varphi_{a_i}}\R^{\Phi(p_i^{k_i})}$ for $2\leq i\leq m$, it follows from Corollary \ref{construccionsuma} that
	\begin{align*}
	d_S(n)&\leq \Phi(2 p_1^{k_1})+\Phi(p_2^{k_2})+\cdots+\Phi(p_m^{k_m})+1\\
	&=\Phi(p_1^{k_1})+\Phi(p_2^{k_2})+\cdots+\Phi(p_m^{k_m})+1\\
	&=\Phi(p_1^{k_1}p_2^{k_2}\cdots p_m^{k_m})+1\\
	&=\Phi(2p_1^{k_1} p_2^{k_2}\cdots p_m^{k_m})+1\\
	&=d(n).
	\end{align*}
	Therefore, $d_S(n)=d(n)$. 
	
	In conclusion, $d_S(n)=d(n)$ for all $n\geq 3$. 
\end{proof}

In other words, this theorem says that the compact flat manifold with holonomy group $\Z_n$ of minimal dimension can be chosen as a solvmanifold.

\medskip

As a consequence of the proof of Theorem \ref{mindim} we obtain the following result, which provides another proof of Theorem \ref{holabelianofinito} for cyclic groups.
 
\begin{coro}
Given $n\in\N, n\geq 2$, there exists a flat almost abelian solvmanifold which has holonomy group $\Z_n$.
\end{coro}

\medskip

\section{Low dimensions}
In this section we will talk about low-dimensional solvmanifolds which admit a flat metric. We are interested in the possible holonomy groups, to see a classification of low-dimensional flat solvmanifolds we refer to \cite{Mor}.

Note that an abelian Lie algebra gives rise to a simply connected abelian Lie group and thus every solvmanifold obtained as a quotient of an abelian Lie group is diffeomorphic to a torus which has trivial holonomy. Therefore, we will not consider abelian Lie algebras.

Let $\g=\z(\g)\oplus\b\oplus [\g,\g]$ be a flat Lie algebra. Recall from Theorem \ref{alglieplanas} that $\dim[\g,\g]$ is even, so that if $\g$ is not abelian, then $\dim\g>2$.

\subsection{Dimension 3}
The only non abelian possibility is $\dim\b=1$ and $\dim[\g,\g]=2$, so that $\g$ is almost abelian according to Theorem \ref{alglieplanas} and hence $\g=\R x\ltimes_{\ad_x}\R^2$ where $\ad_x:\R^2\to\R^2$ is skew-adjoint and invertible. There exists an orthonormal basis $\mathcal{B}=\{e_1,f_1\}$ of $\R^2$ such that \[ [\ad_x]_{\B}=\matriz{0&a\\-a&0},\quad\text{with}\;a\in\R-\{0\}.\]
We will denote by $\g_a$ this flat Lie algebra. It is easy to see that $\g_a$ is isometrically isomorphic to $\g_b$ if and only if $b=\pm a$.  However they are all isomorphic as Lie algebras. Since the existence of lattices only depends on the Lie group structure we may assume $a=1$. The Lie algebra $\g_1$ is isomorphic to the Lie algebra of the group of rigid motions of the Euclidean plane and traditionally it is denoted by $\e(2)$, i.e.
 \[\e(2)=\R x\ltimes_{\ad_x}\R^2,\] with brackets given by
 \[[x,e_1]=-f_1,\quad [x,f_1]=e_1.\]
 
 \bigskip
 
 The simply connected Lie group associated to $\e(2)$ is 
 \begin{equation}\label{eq18}
 E(2)=\R\ltimes_\varphi \R^2,\qquad\varphi(t)=e^{t \ad_x}=\begin{pmatrix}
 \cos(t)&\sin(t)\\
 -\sin(t)&\cos(t)
 \end{pmatrix}.
 \end{equation}
 
 \begin{prop}\label{conjentera3}
 For $t_0\neq 0$ the matrix $\varphi(t_0)$ is conjugate to an integer matrix if and only if $t_0\in\{\frac{\pi}{2}, \frac{3\pi}{2}, \frac{2\pi}{3}, \frac{4\pi}{3}, \frac{\pi}{3}, \frac{5\pi}{3}, \pi, 2\pi\}+2\pi\Z$.
 \end{prop}
 
 \begin{proof}[Proof:\nopunct]
  If $\varphi(t_0)$ is conjugate to an integer matrix, then \[2\cos(t_0)=\tr(\varphi(t_0))=k,\quad\text{for some}\,k\in\Z.\] As $-2\leq 2\cos(t_0)\leq 2$, the only possibilities for $k$ are $k=0,\pm 1,\pm 2$. It follows that, $\cos(t_0)=0,\pm \frac{1}{2},\pm 1$ which implies $t_0\in\{\frac{\pi}{2}, \frac{3\pi}{2}, \frac{2\pi}{3}, \frac{4\pi}{3}, \frac{\pi}{3}, \frac{5\pi}{3}, \pi, 2\pi\}+2\pi\Z$.
 
 \medskip
 
 Conversely, if $t_0\in\{\frac{\pi}{2},\frac{3\pi}{2},\pi,2\pi\}+2\pi\Z$ then $\varphi(t_0)$ is already an integer matrix. Suppose now that $t_0\in\{\frac{2\pi}{3}, \frac{4\pi}{3}, \frac{\pi}{3}, \frac{5\pi}{3}\}+2\pi\Z$. The eigenvalues of $\varphi(t_0)$ are $e^{\pm i t_0}$, which are different. Therefore $\varphi(t_0)$ is conjugate to the companion matrix of its characteristic polynomial, $p(\lambda)=\lambda^2-2\lambda\cos(t_0)+1$, which has integer coefficients because $2\cos(t_0)=\pm 1$ where the sign depends on $t_0$. 
 \end{proof}
 
 \begin{coro}\label{hol3}
 	The possible holonomy groups of three-dimensional flat solvmanifolds are $\{e\}, \Z_2, \Z_3, \Z_4$ and $\Z_6$.
 \end{coro}
 \begin{proof}[Proof:\nopunct]
 	 All the three-dimensional flat solvmanifolds are obtained as quotients of the flat almost abelian Lie group $E(2)$ so the corollary follows by using Theorem \ref{holonomia_lcm} for each of the values of $t_0$ obtained in Proposition \ref{conjentera3}. 
 \end{proof}
 
 The classification of three dimensional compact flat manifolds follows from the classification of the Bieberbach groups in dimension 3, which was completed on 1935 by W. Hantzsche and H. Wendt \cite{HW}. In Wolf's book \cite{Wo}, the ten compact flat manifolds of dimension 3 (up to diffeomorphism) are listed, where six of them are orientable and four are not. Recently, Conway and Rossetti in \cite{CR} studied in detail these manifolds, giving a new systematic manner to describe some properties. They propose the names \textit{platycosms}. The symbols and the names are the following:
 
 \begin{center}
 	\begin{tabular}{|c|c|c|c|c|c|}
 		\hline
 		Name \cite{Wo} & $\hol(M)$& $H_1(M,\Z)$&Orientable&Symbol\cite{CR}&Name\cite{CR} \\
 		\hline \rule{0pt}{2.5ex}
 		$\mathcal{G}_1$&$\{e\}$&$\Z\oplus\Z\oplus\Z$&$\checkmark$&$c1$&\textit{torocosm}\\ \rule{0pt}{2.5ex}
 		$\mathcal{G}_2$&$\Z_2$&$\Z\oplus\Z_2\oplus\Z_2$&$\checkmark$&$c2$&\textit{dicosm}\\ \rule{0pt}{2.5ex}
 		$\mathcal{G}_3$&$\Z_3$&$\Z\oplus\Z_3$&$\checkmark$&$c3$&\textit{tricosm}\\ \rule{0pt}{2.5ex}
 		$\mathcal{G}_4$&$\Z_4$&$\Z\oplus\Z_2$&$\checkmark$&$c4$&\textit{tetracosm}\\ \rule{0pt}{2.5ex}
 		$\mathcal{G}_5$&$\Z_6$&$\Z$&$\checkmark$&$c6$&\textit{hexacosm}\\ \rule{0pt}{2.5ex}
 		$\mathcal{G}_6$&$\Z_2\oplus\Z_2$&$\Z_4\oplus\Z_4$&$\checkmark$&$c22$&\textit{didicosm}\\ [0.5ex] \hline \rule{0pt}{2.5ex}
 		$\mathcal{B}_1$&$\Z_2$&$\Z\oplus\Z\oplus\Z_2$&$\times$&$+a1$&\textit{first amphicosm}\\ \rule{0pt}{2.5ex}
 		$\mathcal{B}_2$&$\Z_2$&$\Z\oplus\Z$&$\times$&$-a1$&\textit{second amphicosm}\\ \rule{0pt}{2.5ex}
 		$\mathcal{B}_3$&$\Z_2\oplus\Z_2$&$\Z\oplus\Z_2\oplus\Z_2$&$\times$&$+a2$&\textit{first amphidicosm}\\ \rule{0pt}{2.5ex}
 		$\mathcal{B}_4$&$\Z_2\oplus\Z_2$&$\Z\oplus\Z_4$&$\times$&$-a2$&\textit{second amphidicosm}\\ [0.5ex]
 		\hline
 	\end{tabular}
 \end{center}	
 
 \medskip
 
 We can identify which of these compact flat manifolds of dimension 3 can be seen as flat solvmanifolds. Indeed, as every solvmanifold is parallelizable, in particular orientable, we discard $\B_1, \B_2, \B_3$ and $\B_4$. Since, according to Theorem \ref{holonomia_lcm}, a flat almost abelian solvmanifold has cyclic holonomy group, the Hantzsche-Wendt manifold, $\G_6$, cannot be obtained as a flat solvmanifold either. Moreover, using Proposition \ref{bock} it can be seen that there exist lattices $\Gamma_i\subset E(2)$ for $i=1,\ldots, 5$ such that $\varGamma_i\backslash E(2)$ is diffeomorphic to $\G_i$.
 
\medskip

\subsection{Dimension 4}
Let $\g=\z(\g)\oplus\b\oplus [\g,\g]$ a non abelian flat Lie algebra of dimension 4. Evidently, the only possibility is $\dim[\g,\g]=2$, $\dim\b=1$ and $\dim\z(\g)=1$. We can write $\g$ as $\g=\R x\ltimes_{\ad_x} \R^3$, where $[\ad_x]_{\mathcal{B}}=\matriz{0&0&0\\0&0&1\\0&-1&0}$ for some basis $\mathcal{B}=\{z,e_1,f_1\}$ of $\R^3$. Then the simply connected Lie group is \begin{equation}\label{ecuación}
G=\R\ltimes_\varphi \R^3,\quad\text{with}\;
\varphi(t)=\matriz{1&0&0\\0&\cos(t)&\sin(t)\\0&-\sin(t)&\cos(t)}.
\end{equation}

\medskip

As $G$ is almost abelian, we can use Proposition \ref{bock} to determine its lattices.

\medskip

\begin{lema}\label{conjentera4}
	For $t_0\neq 0$ the matrix $\varphi(t_0)$ is conjugate to an integer matrix if and only if 
	$t_0\in\big\{\frac{\pi}{2}, \frac{3\pi}{2}, \frac{2\pi}{3}, \frac{4\pi}{3}, \frac{\pi}{3}, \frac{5\pi}{3}, \pi, 2\pi\big\}+2\pi\Z$.
	\end{lema}
	
	\begin{proof}[Proof:\nopunct]
		If $\varphi(t_0)$ is conjugate to an integer matrix then $1+2\cos(t_0)=\tr(\varphi(t_0))=k$ for some $k\in\Z$. Therefore $2\cos(t_0)\in\Z$, which implies that $t_0$ is as in the list above.
		
		Conversely, if $t_0$ belongs to the list above, it follows from Proposition \ref{conjentera3} that the matrix $\matriz{\cos(t_0)&\sin(t_0)\\-\sin(t_0)&\cos(t_0)}$ is conjugate to an integer matrix through a matrix $P\in\GL(2,\R)$. We choose $\tilde{P}=\matriz{1&0\\0&P}$, so $\tilde{P}^{-1}\varphi(t_0)\tilde{P}$ is integer.
		\end{proof}

\begin{coro}\label{possible4}
	The possible holonomy groups of four-dimensional flat solvmanifolds are $\{e\}, \Z_2, \Z_3, \Z_4$ and $\Z_6$.
\end{coro}		

\medskip

Using Proposition \ref{bock} we can exhibit flat Lie groups with lattices such that the holonomy groups of the associated solvmanifolds are all the groups in the Corollary \ref{possible4}.

\subsection{Dimension 5}
Let $\g=\z(\g)\oplus\b\oplus[\g,\g]$ be a non abelian flat Lie algebra of dimension 5, then there are two possibilities: $\dim[\g,\g]=2$ or $\dim[\g,\g]=4$. 

\medskip

$\bullet$ If $\dim[\g,\g]=2$ then $\dim\b=1$ and $\dim\z(\g)=2$ so that we can write\footnote{We will denote by $A=A_1\oplus A_2$ the block diagonal matrix in $\R^4$ given by $A=\matriz{A_1&\\&A_2}$.} \[\g=\R x\ltimes_{\ad_x} \R^4,\quad\text{where}\; [\ad_x]_{\mathcal{B}}=\matriz{0&0\\0&0}\oplus\matriz{0&1\\-1&0},\] for some basis $\mathcal{B}$ of $\R^4$. Therefore the corresponding simply connected Lie group is  \[G=\R\ltimes_{\tilde{\varphi}}\R^4,\quad\text{where}\;\tilde{\varphi}(t)=\matriz{1&0\\0&1}\oplus\matriz{\cos(t)&\sin(t)\\-\sin(t)&\cos(t)}\]  

It can be deduced, analogously to Proposition \ref{conjentera4}, that the only values of $t_0$ for which $\R\ltimes_{\tilde{\varphi}} \R^4$ admits lattices are $2\pi,\pi,\frac{2\pi}{3},\frac{\pi}{2}$ and $\frac{\pi}{3}$. Therefore the holonomy groups in this case will be $\{e\}, \Z_2, \Z_3, \Z_4$ and $\Z_6$. 

\medskip

$\bullet$ Let $\dim[\g,\g]=4$, so that $\dim\b=1$ and $\z(\g)=0$, hence we can write \[\g=\R x\ltimes_{\ad_x}\R^4,\quad\text{where}\;[\ad_x]_{\mathcal{B}}=\matriz{0&a\\-a&0}\oplus\matriz{0&b\\-b&0}\quad\text{with}\;a\geq b>0\] for some basis $\mathcal{B}$ of $\R^4$. We will denote this Lie algebra by $\g_{a,b}$. 
%Ya hemos observado  más generalmente que $\g_b$ es isomorfa a $\g_b'$ con $0<b\leq 1$ y $0<b'\leq 1$ si y sólo si $b=b'$. 
The corresponding simply connected Lie group is
\begin{equation}\label{eq16}
G_{a,b}=\R\ltimes_{\varphi_{a,b}}\R^4,\;\text{with}\;\varphi_{a,b}(t)=\matriz{\cos(at)&\sin(at)\\-\sin(at)&\cos(at)}\oplus\matriz{\cos(bt)&\sin(bt)\\-\sin(bt)&\cos(bt)}.
\end{equation}

\medskip

\begin{lema}\label{conjentera5}
	For $t_0\neq0$, $\varphi_{a,b}(t_0)$ is conjugate to an integer matrix if and only if one of the following cases occurs\footnote{The values for $at_0$ and $bt_0$ are interchangeable.}:
	\medskip
	
	\begin{center}
		\begin{tabular}{|c|c|c|}
			%\hline
			%\multicolumn{3}{|c|}{Posibilidades} \\
			\hline
			& Values for $at_0$ & Values for $bt_0$\\\hline
			\rule{0pt}{3ex} {Case $(1)$}&$2\pi\Z$& $\{\frac{\pi}{2},\frac{3\pi}{2},\frac{2\pi}{3},\frac{4\pi}{3},\frac{\pi}{3},\frac{5\pi}{3},\pi,2\pi\}+2\pi\Z$\\[1ex] \hline
		\end{tabular}
		
		\smallskip
		
		\begin{tabular}{|c|c|c|}\hline
			& Values for $at_0$ & Values for $bt_0$\\\hline
			\rule{0pt}{3ex} {Case $(2)$}& $\{\frac{\pi}{2},\frac{3\pi}{2},\frac{2\pi}{3},\frac{4\pi}{3}, \frac{\pi}{3},\frac{5\pi}{3}, \pi\}+2\pi\Z$ & $\{t_0,-t_0\}+2\pi\Z$ \\ [1ex] \hline
		\end{tabular}
		
		\smallskip
		
		\begin{tabular}{|c|c|c|}\hline
			& Values for $at_0$ & Values for $bt_0$\\\hline
			\rule{0pt}{3ex}\multirow{3}{*}{Case $(3)$} &$\{\frac{\pi}{2},\frac{3\pi}{2},\frac{2\pi}{3},\frac{4\pi}{3},\frac{\pi}{3},\frac{5\pi}{3}\}+2\pi\Z$ &  $\pi+2\pi\Z$  \\\rule{0pt}{3ex}
			& $ \{\frac{\pi}{2},\frac{3\pi}{2}\}+2\pi\Z$ & $\{\frac{2\pi}{3},\frac{4\pi}{3}\}+2\pi\Z$ \\\rule{0pt}{3ex}
			& $\{\frac{\pi}{3},\frac{5\pi}{3}\}+2\pi\Z$& $ \{\frac{\pi}{2},\frac{3\pi}{2},\frac{2\pi}{3},\frac{4\pi}{3}\}+2\pi\Z$\\[1ex]\hline
		\end{tabular}
		
		\smallskip
		
		\begin{tabular}{|c|c|c|}\hline
			& Values for $at_0$ & Values for $bt_0$\\\hline
			\rule{0pt}{3ex} \multirow{4}{*}{Case $(4)$}
			&$\{\frac{\pi}{4},\frac{7\pi}{4}\}+2\pi\Z$&$\{\frac{3\pi}{4},\frac{5\pi}{4}\}+2\pi\Z$\\\rule{0pt}{3ex}
			&$\{\frac{\pi}{5},\frac{9\pi}{5}\}+2\pi\Z$&$\{\frac{3\pi}{5},\frac{7\pi}{5}\}+2\pi\Z$\\\rule{0pt}{3ex}
			&$\{\frac{2\pi}{5},\frac{8\pi}{5}\}+2\pi\Z$&$\{\frac{4\pi}{5},\frac{6\pi}{5}\}+2\pi\Z$\\\rule{0pt}{3ex}
			&$\{\frac{\pi}{6},\frac{11\pi}{6}\}+2\pi\Z$&$\{\frac{5\pi}{6},\frac{7\pi}{6}\}+2\pi\Z$\\ [1ex]
			\hline
		\end{tabular}	
	\end{center}
\end{lema}	
\begin{proof}[Proof:\nopunct]
	If $\varphi_{a,b}(t_0)$ is conjugate to an integer matrix then its characteristic polynomial, $P(\lambda)$, must have integer coefficients. Let $x=\cos(at_0)$ and $y=\cos(bt_0)$. We have \[P(\lambda)=\lambda^4-2(x+y)\lambda^3+(2+4xy)\lambda^2-2(x+y)\lambda+1,\] hence we have the system 
	\[\begin{cases}
	2(x+y)=m\\
	4xy=n
	\end{cases},\] for some integers $m$ and $n$. Solving it, we get
\[\cos(at_0)=x=\frac{m\pm\sqrt{m^2-4n}}{4}\quad\text{and}\;\cos(bt_0)=y=\frac{m\mp\sqrt{m^2-4n}}{4}.\] 

	Suppose first that $\cos(at_0)=\frac{m+\sqrt{m^2-4n}}{4}$, so that $\cos(bt_0)=\frac{m-\sqrt{m^2-4n}}{4}$. From the fact that  $m=2(\cos(at_0)+\cos(bt_0)))$, it follows that $|m|\leq 4$. On the other hand given that $|\cos(at_0)|\leq 1$ and $|\cos(bt_0)|\leq 1$ we have $|m+\sqrt{m^2-4n}|\leq 4$ and $|m-\sqrt{m^2-4n}|\leq 4$. Then  \[-4-m\leq\sqrt{m^2-4n}\leq4-m\quad\text{and}\; 4+m\geq\sqrt{m^2-4n}\geq m-4.\] 
	As $|m|\leq 4$, these conditions translate into \[0\leq \sqrt{m^2-4n}\leq4-m\quad\text{and}\quad 0\leq\sqrt{m^2-4n}\leq 4+m.\]
	
	Giving all the different values to $m$ we obtain the corresponding values for $n$ and finally the possible choices for $(m,n)$ are 
	\[(0,n)\;\text{for}\, -4\leq n\leq 0,\;(\pm1,-2),(\pm1,-1),(\pm1,0),(\pm2,0),(\pm2,1),(\pm3,2),(\pm4,4).\] The corresponding values for $\cos(at_0)$ and $\cos(bt_0)$ are
	\begin{center}
		\begin{tabular}{|c|c|c|c|c|c|}
			\hline
			\multicolumn{6}{|c|}{Case $(1)$}\\
			\hline
			&$(2,0)$&$(1,-2)$&$(3,2)$&$(0,-4)$&$(4,4)$\\ \hline\rule{0pt}{2ex}
			
			$\cos(at_0)$&$1$&$1$&$1$&$1$&$1$\\
			
			$\cos(bt_0)$&$0$&$-\frac{1}{2}$&$\frac{1}{2}$&$-1$&$1$\\[0.5ex]\hline
		\end{tabular}
		
		\smallskip
		
		\begin{tabular}{|c|c|c|c|c|}
			\hline
			\multicolumn{5}{|c|}{Case $(2)$}\\
			\hline
			&$(0,0)$&$(-2,1)$&$(2,1)$&$(-4,4)$\\\hline \rule{0pt}{2.5ex}
			
			$\cos(at_0)$&$0$&$-\frac{1}{2}$&$\frac{1}{2}$&$-1$\\[0.5ex]
			
			$\cos(bt_0)$&$0$&$-\frac{1}{2}$&$\frac{1}{2}$&$-1$\\[0.5ex]\hline
		\end{tabular}
		
		\smallskip
		
		\begin{tabular}{|c|c|c|c|c|c|c|}
			\hline
			\multicolumn{7}{|c|}{Case $(3)$}\\
			\hline
			&$(-2,0)$ &$(-3,2)$&$(-1,-2)$&$(-1,0)$&$(1,0)$&$(0,-1)$\\\hline \rule{0pt}{2.5ex}
			$\cos(at_0)$&$0$&$-\frac{1}{2}$&$\frac{1}{2}$&$0$&$\frac{1}{2}$&$\frac{1}{2}$\\[0.5ex]
			$\cos(bt_0)$&$-1$&$-1$&$-1$&$-\frac{1}{2}$&$0$&$-\frac{1}{2}$\\[0.5ex]\hline
		\end{tabular}
		
		\smallskip
		\begin{tabular}{|c|c|c|c|c|}
			\hline
			\multicolumn{5}{|c|}{Case $(4)$}\\
			\hline
			&$(0,-2)$&$(1,-1)$&$(-1,-1)$&$(0,-3)$\\ \hline \rule{0pt}{2.5ex}
			$\cos(at_0)$&$\frac{\sqrt{2}}{2}$&$\frac{1+\sqrt{5}}{4}$&$\frac{-1+\sqrt{5}}{4}$&$\frac{\sqrt{3}}{2}$\\[0.5ex]
			$\cos(bt_0)$&$-\frac{\sqrt{2}}{2}$&$\frac{1-\sqrt{5}}{4}$&$-\frac{1+\sqrt{5}}{4}$&$-\frac{\sqrt{3}}{2}$\\[0.5ex]\hline
		\end{tabular}
	\end{center}
	If we suppose that $\cos(at_0)=\frac{m-\sqrt{m^2-4n}}{4}$ and $\cos(bt_0)=\frac{m+\sqrt{m^2-4n}}{4}$ then we will only exchange the values for $at_0$ and $bt_0$ in the tables of the cases $(1),  (2), (3)$ and $(4)$, respectively. We have therefore obtained the list in the statement.
	
	Conversely, for the cases $(1), (2)$ and $(3)$, we use Proposition \ref{conjentera3} and we conjugate the matrix $\varphi_{a,b}(t_0)$ to an integer matrix through a block matrix. For the case $(4)$, as all the eigenvalues of $\varphi_{a,b}(t_0)$ are different, we may conjugate $\varphi_{a,b}(t_0)$ to the companion matrix of its characteristic polynomial which has integer coefficients as can be seen.
\end{proof}

\medskip

\begin{teo}\label{hol5}
	The list of all the possible holonomy groups of flat solvmanifolds of dimension 5 is the following: 
	\[\{e\}, \Z_2,\Z_3,\Z_4,\Z_5,\Z_6,\Z_8,\Z_{10},\Z_{12}.\] 
\end{teo}
\begin{proof}[Proof:\nopunct]
	The result follows by analyzing all the cases obtained in Lemma \ref{conjentera5} and using Theorem \ref{holonomia_lcm}. 
\end{proof}

Again, using Proposition \ref{bock} we can exhibit flat Lie groups with lattices such that the holonomy groups of the associated solvmanifolds are all the groups in Theorem \ref{hol5}.

\medskip 

\subsection{Dimension 6}
Let $\g=\z(\g)\oplus\b\oplus[\g,\g]$ a non abelian flat Lie algebra of dimension 6. There are two possibilities for $\dim[\g,\g]$, namely $\dim[\g,\g]=2$ or $\dim[\g,\g]=4$. 

\medskip

If $\dim[\g,\g]=2$, then $\dim\b=1$ (recall that in general $\dim\b\leq\frac{\dim[\g,\g]}{2}$) and thus $\dim\z(\g)=3$. Therefore the simply connected Lie group corresponding to $\g=\R x\ltimes_\varphi\R^5$ can be written as $G=\R\ltimes_{\tilde{\varphi}} \R^5$, with $\tilde{\varphi}(t)=\Id_3\oplus\matriz{\cos(t)&\sin(t)\\-\sin(t)&\cos(t)}$. It can be easily seen that $\tilde{\varphi}(t_0)$ is conjugate to an integer matrix if and only if
$t_0\in\{\frac{\pi}{2}, \frac{3\pi}{2}, \frac{2\pi}{3}, \frac{4\pi}{3}, \frac{\pi}{3}, \frac{5\pi}{3}, \pi, 2\pi\}+2\pi\Z$. In this case, all the possible holonomy groups are $\{e\},\Z_2,\Z_3,\Z_4,\Z_6$. 

\bigskip

If $\dim[\g,\g]=4$, then $\dim\b$ can be 1 or 2. 

Let $\dim\b=1$, so that $\dim\z(\g)=1$. The corresponding simply connected Lie group is $G=\R\ltimes_{\varphi}\R^5$, where $\varphi(t)=\matriz{1&&\\&\theta(at)&\\&&\theta(bt)}$ with $a\geq b>0$. As in the proof of Lemma \ref{conjentera5}, it can be seen that $\varphi(t_0)$ is conjugate to an integer matrix if and only if $at_0$ and $bt_0$ are as in Lemma \ref{conjentera5}. Therefore, the possible holonomy groups are $\{e\}, \Z_2, \Z_3, \Z_4, \Z_5, \Z_6, \Z_8, \Z_{10}$ and $\Z_{12}$.

\medskip

Let $\dim\b=2$, so that $\z(\g)=0$. Note that in this case $\g$ is not almost abelian, hence we cannot use Proposition \ref{bock} to determine the existence of lattices and the holonomy groups of the corresponding solvmanifolds. However, we will be able to exhibit examples of flat solvmanifolds with non cyclic holonomy group. These examples will be deduced from the following general construction.

\begin{prop}\label{decomposition}
	Let $\g=\z(\g)\oplus\b\oplus[\g,\g]$ be a flat Lie algebra with $\dim\b=\frac{\dim[\g,\g]}{2}$. Then $\g$ is isomorphic as a Lie algebra to $\z(\g)\times \e(2)^{\dim\b}$. 
\end{prop}
\begin{proof}[Proof:\nopunct]
	Let $\dim[\g,\g]=2n$. It follows from Theorem \ref{alglieplanas} that there exists an orthonormal basis $\mathcal{B}=\{f_1,\,\ldots\,,f_{2n}\}$ of $[\g,\g]$ and $\alpha_1,\ldots,\alpha_n\in\b^*$ such that for all $x\in\b$,
	\[
	[\ad_x]_{\mathcal{B}}=\matriz{
		& \alpha_1(x)&&&\\
		-\alpha_1(x)&&&\\
		&&\ddots&&\\
		&&&& \alpha_n(x)\\
		&&&-\alpha_n(x)}
	.\]
	Since $\ad:\b\to\so[\g,\g]$ is injective, we have $\cap_{i=1}^n \ker\alpha_i=\{0\}$. It follows that $\{\alpha_i\}_{i=1}^n$ is a basis of $\b^*$. Now let $\{e_i\}_{i=1}^n$ be the dual basis of $\{\alpha_i\}_{i=1}^n.$ Then \[
	[\ad_{e_i}]_{\mathcal{B}}=\matriz{
		0&&&&&\\
		&\ddots&&&&\\
		&&0&1&&\\
		&&-1&0&&\\
		&&&&\ddots&\\
		&&&&&0}
	,\]
	i.e. the only brackets different from zero are $[e_i,f_{2i-1}]=-f_{2i}$ and $[e_i,f_{2i}]=f_{2i-1}$, for $1\leq i\leq n$. Let $\g_i=\text{span}\{e_i,f_{2i-1},f_{2i}\}$, which is an ideal of $\g$ isomorphic to $\e(2)$ for all $i$. Therefore $\g$ splits as a direct product of Lie algebras, \[\g=\z(\g)\times\g_1\times\cdots\times\g_n\simeq \z(\g)\times\e(2)^n.\] 
\end{proof}

\begin{obs}
	The basis $\{e_i\}_{i=1}^n$ of $\b$ appearing in the proof is not necessarily orthonormal and therefore the decomposition of Theorem \ref{decomposition} is not necessarily orthogonal.
\end{obs}

\medskip

The associated simply connected Lie group is $G=\R^s\times E(2)^n$, where $s=\dim\z(\g)$ and $n=\dim\b$. Let $\varGamma_1, \varGamma_2,\,\ldots\,,\varGamma_n$ be lattices in $E(2)$. We construct a lattice in $G$ choosing the product of the lattices, $\varGamma=\Z^s\times\varGamma_1\times\cdots\times\varGamma_n$. The corresponding solvmanifold, $\varGamma\backslash G\cong \Z^s\backslash\R^s\times\varGamma_1\backslash E(2)\times\cdots\times\varGamma_n\backslash E(2)$ is equipped with a flat metric $g_0$ induced by the flat left invariant metric on $G$. On the other hand, if we consider on $\varGamma\backslash G$ the product metric $g_*$ where each factor $\varGamma_i\backslash E(2)$ carries a flat metric induced by a flat left invariant metric on $E(2)$, then $(\varGamma\backslash G, g_0)$ and $(\varGamma\backslash G, g_*)$ are two compact flat manifolds with isomorphic fundamental groups. By Theorem \ref{Biebvar2} there exists an affine equivalence between them, which implies $\hol(\varGamma\backslash G,g_0)\cong\hol(\varGamma\backslash G,g_*)$, but the last one is isomorphic to $\hol(T^s)\times\hol(\varGamma_1\backslash E(2))\times \cdots\times \hol(\varGamma_n\backslash E(2)).$

In conclusion \[\hol(\varGamma\backslash G)=\Z_{k_1}\oplus\cdots\oplus\Z_{k_n},\quad\text{with}\; k_j=|\Lambda_j\backslash\varGamma_j|,\] where $\Lambda_j$ is the maximal normal abelian subgroup of $\varGamma_j$. Moreover, it follows from Corollary \ref{hol3} that $k_i\in\{2,3,4,6\}$.

In the particular case of dimension 6, with $\g=\b\oplus [\g,\g]$, $\dim[\g,\g]=4$ and $\dim\b=2$, we get the following holonomy groups: 

\begin{equation}\label{eq*}
\begin{split}
\Z_2\oplus\Z_2,\;\; \Z_2\oplus\Z_3,\;\; \Z_2\oplus\Z_4,\;\; \Z_2\oplus\Z_6,\;\; \Z_3\oplus\Z_3, \\
\Z_3\oplus\Z_4,\;\;\Z_3\oplus\Z_6,\;\;\Z_4\oplus\Z_4,\;\;\Z_4\oplus\Z_6,\;\;\Z_6\oplus\Z_6.
\end{split}
\end{equation}
\begin{obs}\label{dim6}
	As mentioned in Corollary \ref{kähler}, every flat solvmanifold of even dimension admits a Kähler structure. In \cite{De}, all the holonomy groups of six-dimensional Kähler compact flat manifolds are determined. They are: the trivial group, $\Z_n$ for $n=2,3,4,5,6,8,10,12$, the direct sums in \eqref{eq*} and the dihedral group $D_8$. It is stated in \cite{De} that there is only one Kähler compact flat manifold in dimension 6 with this holonomy group and this manifold has $\beta_1=0$. This compact flat manifold is not homeomorphic to any solvmanifold according to Remark \ref{primitivo}. In conclusion, the only possible holonomy groups of flat solvmanifolds in dimension 6 are the ones we have found and there are no more.
\end{obs}

\
\end{document}